\documentclass{article}

\usepackage[utf8]{inputenc}
\usepackage[T1]{fontenc}
\usepackage{hyperref}
\usepackage{url}
\usepackage{booktabs}
\usepackage{amsfonts}
\usepackage{nicefrac}
\usepackage{microtype}
\usepackage{lipsum}
\usepackage{graphicx}
\usepackage{amsmath}
\usepackage{amssymb}
\usepackage{amsthm}
\usepackage{amscd}
\usepackage{alltt}
\usepackage{enumitem}
\usepackage[all]{xy}
\SelectTips{cm}{12}

\newtheorem{lemma}{Lemma}
\newtheorem{prop}{Proposition}
\newtheorem{theorem}{Theorem}
\newtheorem*{theorem*}{Theorem}

\newcommand{\rar}{\rightarrow}
\newcommand{\lar}{\leftarrow}

\DeclareMathOperator\mat{M}
\DeclareMathOperator\elin{E}
\DeclareMathOperator\stlin{St}

\DeclareMathOperator\klin{K}
\DeclareMathOperator\stmap{st}
\DeclareMathOperator\glin{GL}
\DeclareMathOperator\diag{D}
\DeclareMathOperator\Jac{J}
\newcommand{\leqt}{\trianglelefteq}
\DeclareMathOperator{\Ad}{Ad}

\newcommand{\up}[2]{{^{#1}\!{#2}}}

\newcommand{\Set}{\mathbf{Set}}
\newcommand{\Group}{\mathbf{Grp}}
\DeclareMathOperator{\Pro}{Pro}

\title{Centrality of \(K_2\)-functor revisited}

\author{
  Egor Voronetsky
  \thanks{Research is supported by the Russian Science Foundation grant 19-71-30002.} \\
  Chebyshev Laboratory, \\
  St. Petersburg State University, \\
  14th Line V.O., 29B, \\
  Saint Petersburg 199178 Russia \\
}

\begin{document}
\maketitle

\begin{abstract}
We prove that \(\stlin(n, A)\) is a crossed module over \(\glin(n, A)\) under a local stable rank condition on an algebra \(A\) over a commutative ring. Our proof uses only elementary localization techniques in terms of pro-groups and stability results for \(\klin_1\) and \(\klin_2\). We also prove similar result for the Steinberg group associated with any sufficiently isotropic general linear group constructed by a quasi-finite algebra.
\end{abstract}

\section{Introduction}

In \cite{AnotherPresentation} W. van der Kallen proved that the Steinberg group \(\stlin(n, K)\) is a crossed module over \(\glin(n, K)\) for any commutative ring \(K\) and for any \(n \geq 4\). Together with the fact that \(\stlin(n, A)\) is centrally closed for any associative ring \(A\) and for \(n \geq 5\) this gives explicit description of the Schur multiplier of the elementary group \(\elin(n, K)\) for \(n \geq 5\).

Later M.\,S. Tulenbaev extended this result to all finite \(K\)-algebras, see \cite{Tulenbaev}. Both these proofs use the so-called ``another presentation'' of the Steinberg group in terms of arbitrary transvections instead of the elementary ones. Recently in \cite{CentralityC, CentralityD, CentralityE} S. Sinchuk and A. Lavrenov proved centrality of \(\klin_2\)-functors for Chevalley groups of types \(C_l\), \(D_l\), and \(E_l\) using similar methods and the linear case. In 1998 A. Bak and G. Tang announced centrality for the even unitary case, but their proof was not published.

In the local case much more is known. From surjective stability for \(\klin_2\) (see \cite{LinStabDennis}) it follows that \(\klin_2(n, A)\) is central if \(n \geq \mathrm{sr}(A) + 2\). In particular, \(\klin_2(n, A)\) is central for all \(n \geq 3\) if \(A\) is semi-local. In \cite{Stavrova} A. Stavrova proved centrality for all isotropic reductive groups over local rings (with isotropic rank at least \(2\)).

In \cite{PetrovStavrova} V. Petrov and A. Stavrova defined isotropic reductive groups and proved that the elementary subgroup is normal is such groups. For non-commutative rings we may similarly define an isotropic general linear group as \(\glin(R) = R^*\) if \(R\) contains a complete family of orthogonal idempotents satisfying an additional condition (``Morita equivalence'', i.e. every idempotent is full). These groups are more general than the usual matrix groups with \(R = \mat(n, A)\), but they still have elementary subgroups and Steinberg groups. It seems that it is easy to generalize some fundamental results about matrix groups to the isotropic case: the elementary subgroup is normal in \(\glin(R)\), there is a description of subgroups of \(\glin(R)\) normalized by \(\elin(R)\) in terms of ideals of \(R\), and the Steinberg group is universally closed (see, for example, \cite{HahnOMeara} for the matrix case). Usual proofs of these facts do not really use the whole structure of a matrix ring.

On the other hand, it is unclear even how to formulate ``another presentation'' in the isotropic case for non-commutative rings. In this paper we give a proof that the Steinberg group is indeed a crossed module. Our method uses an action of the local general linear group \(\glin(S^{-1} R)\) on the Steinberg pro-group \(\stlin(R)^{(\infty)}\), where \(S\) is some commutative multiplicative subset. This pro-group is the formal projective limit of ``homotopes'' \(\stlin(R)^{(s)}\) for all \(s \in S\). Our main result is
\begin{theorem*}
Let \(K\) be a commutative ring, \(R\) be a unital \(K\)-algebra with a complete family of Morita equivalent orthogonal idempotents \(e_1, \ldots, e_n\) for \(n \geq 4\). Suppose that either \(R = \mat(n, A)\) for an algebra \(A\) with \(\mathrm{lsr}(A) \leq n - 2\) (and \(e_i = e_{ii}\) are the matrix units) or \(R\) is quasi-finite. Then there is unique action of \(\glin(R)\) on \(\stlin(R)\) making \(\stmap \colon \stlin(R) \rar \glin(R)\) a crossed module.
\end{theorem*}

We also prove similar result for semi-local \(R\) and for \(n \geq 3\).

In section 2 we recall the definition of Steinberg groups and show how they are related under changes of the family of idempotents such as \(e_1, \ldots, e_n \mapsto e_1, \ldots, e_{n - 2}, e_{n - 1} + e_n\). In section 3 we construct the pro-ring \(R^{(\infty)}\) of homotopes of \(R\) and in section 4 we do the same for the Steinberg group. Section 5 contains a construction of the action of \(\glin(S^{-1} R)\) on the Steinberg pro-group. In the last section we prove the main theorem.

The author wants to thank Sergey Sinchuk for motivation and helpful discussions.

\section{Steinberg groups}

We denote the category of sets by \(\Set\) and the category of groups by \(\Group\). If \(\mathcal C\) is a category, then \(\mathcal C(X, Y)\) is the set of arrows from an object \(X\) to an object \(Y\), \(\mathcal C(X)\) is the set of arrows from \(X\) to itself. Compositions are written from the right to the left, i.e. as \(\mathcal C(Y, Z) \times \mathcal C(X, Y) \rar \mathcal C(X, Z)\). We write \(X \in \mathcal C\) if \(X\) is an object of a category \(\mathcal C\).

For group operations we use the conventions \(\up xy = x y x^{-1}\), \(x^y = y^{-1} x y\), and \([x, y] = xyx^{-1} y^{-1}\). There are useful commutator identities
\[
[xy, z] = \up x{[y, z]}\, [x, z], \quad
[x, yz] = [x, y]\, \up y{[x, z]},
\]
and the Hall -- Witt identity
\[[\up z x, [y, z]]\, [\up y z, [x, y]]\, [\up xy, [z, x]] = 1.\]
In particular, if \(x\) and \(z\) commute, then
\[[yx, z] = [y, z], \quad [x, yz] = [x, y],\]
and if \([x, z]\) lies in the center of the group, then
\[[x, [y, z]] = [[x, y], \up yz].\]

If \(G\) and \(H\) are subgroup of some group, then \([G, H] = \langle [g, h] \mid g \in G, h \in H \rangle\) is their commutant and \(\up GH = [G, H]\, H = \langle \up gh \mid g \in G, h \in H \rangle\). Note that \(H\) normalizes \([G, H]\) by the commutator identities.

A precrossed module is a group homomorphism \(d \colon H \rar G\) with an action of \(G\) on \(H\) by automorphisms (the action is denoted by \(\up gh\)) such that \(d(\up gh) = \up g{d(h)}\). A crossed module is a precrossed module with the property \(\up{d(h)}{h'} = \up h{h'}\). For every crossed module the kernel of \(d\) is a central subgroup of \(H\) and the image is a normal subgroup of \(G\).

Every ring in out paper is associative, but not necessarily unital. An element \(x\) in a non-unital ring \(R\) is called quasi-invertible if it is invertible under the monoidal operation \(x \circ y = xy + x + y\). If \(R\) has an identity, then the group of quasi-invertible elements of \(R\) is isomorphic to the group \(R^*\) of invertible elements by \(x \mapsto x + 1\).

A unital ring \(R\) is called semi-simple if it is a finite product of matrix rings over division rings (of finite sizes). A unital ring \(R\) is called semi-local if \(R / \Jac(R)\) is semi-simple, where \(\Jac(R)\) is the Jacobson radical of \(R\). A unital algebra \(R\) over a commutative ring \(K\) is called finite if it is a finitely generated \(K\)-module and quasi-finite if it is a direct limit of finite algebras. Base changes of finite and quasi-finite algebras are finite and quasi-finite. Any finite algebra over a local commutative ring is semi-local. These properties are preserved under Morita equivalence.

From now on fix a unital ring \(R\). We think about \(R\) as a matrix ring. Its general linear group is the group of invertible elements, \(\glin(R) = R^*\).

We say that an idempotent \(e \in R\) Morita dominates an idempotent \(\widetilde e\) if \(\widetilde e \in ReR\). This is equivalent to the following: \(\widetilde e R e\) is a finite projective left \(eRe\)-module, \(e R \widetilde e\) is a finite projective right \(eRe\)-module dual to \(\widetilde e R e\) via the map \(e R \widetilde e \times \widetilde e R e \rar e R e\), and \(\widetilde e R \widetilde e = \widetilde e R e \otimes_{eRe} eR\widetilde e\) is their endomorphism ring. Clearly, Morita dominance is a pre-order relation on idempotents. Idempotents \(e\) and \(\widetilde e\) are called Morita equivalent if they Morita dominate each other.

If \(e\) and \(\widetilde e\) are orthogonal idempotents in \(R\), then \(e + \widetilde e\) Morita dominates both of them. If \(e\) and \(\widetilde e\) are orthogonal and Morita equivalent, then they are Morita equivalent to \(e + \widetilde e\). From now on fix a complete family \(e_1, \ldots, e_n\) of Morita equivalent orthogonal idempotents in \(R\) for some \(n \geq 1\). We use the notation \(R_{ij} = e_i R e_j\).

For all distinct \(i, j \in \{1, \ldots, n\}\) there are natural group homomorphisms \(t_{ij} \colon R_{ij} \rar \glin(R), x \mapsto 1 + x\). The elements \(t_{ij}(a)\) are called elementary transvections. For example, if \(R = \mat(n, A)\) is a matrix ring and \(e_i = e_{ii}\) are the matrix units, then \(t_{ij}(a)\) are ordinary elementary transvections. In any case they satisfy the Steinberg relations
\begin{enumerate}[label = (St\arabic*)]
\item \(t_{ij}(a)\, t_{ij}(b) = t_{ij}(a + b)\);
\item \([t_{ij}(a), t_{kl}(b)] = 1\) for \(j \neq k\) and \(i \neq l\);
\item \([t_{ij}(a), t_{jk}(b)] = t_{ik}(ab)\) for \(i \neq k\).
\end{enumerate}

A Steinberg group \(\stlin(R)\) is the abstract group generated by symbols \(x_{ij}(a)\) for all distinct \(i, j \in \{1, \ldots, n\}\) and all \(a \in R_{ij}\). The relations on these symbols are the Steinberg relations for \(x_{ij}\) instead of \(t_{ij}\). Clearly, \(x_{ij} \colon R_{ij} \rar \stlin(R)\) is a group homomorphism for any \(i, j\), its image is denoted by \(x_{ij}(*)\) and called a root subgroup. There is a homorphism \(\stmap \colon \stlin(R) \rar \glin(R), x_{ij}(a) \mapsto t_{ij}(a)\). Its image (i.e. the group generated by all elementary transvections) is called an elementary group and denoted by \(\elin(R)\), its kernel is denoted by \(\klin_2(R)\), and its cokernel is denoted by \(\klin_1(R)\) (it is the set of cosets in general). Since homomorphisms \(t_{ij}\) are injective, it follows that the root subgroups \(x_{ij}(*)\) are isomorphic to \(R_{ij}\).

A diagonal group is the group \(\diag(R) = \{g \in \glin(R) \mid e_i g e_j = 0 \text{ for } i \neq j\}\) (it is indeed closed under inversion). This group decomposes as a direct product of \(d_i(*)\) for all \(1 \leq i \leq n\), where \(d_i(a) = 1 + a - e_n\) for \(a \in (R_{ii})^*\) and \(d_i(*)\) is the image of the injective group homomorphism \(d_i \colon (R_{ii})^* \rar \glin(R)\). The diagonal group acts on the Steinberg group by
\begin{itemize}
\item \(\up{d_i(a)}{x_{jk}(b)} = x_{jk}(b)\) for \(k \neq i \neq j\);
\item \(\up{d_i(a)}{x_{ij}(b)} = x_{ij}(ab)\);
\item \(\up{d_i(a)}{x_{ji}(b)} = x_{ji}(ba^{-1})\);
\end{itemize}
where \(a^{-1}\) is the inverse of \(a\) in the ring \(R_{ii}\). The homomorphism \(\stmap\) preserves the action of \(\diag(R)\).

If \(R = \mat(n, A)\) is the matrix algebra and \(e_i\) are the matrix units, then the groups defined so far coincide with the usual \(\glin(n, A)\), \(\stlin(n, A)\), \(\elin(n, A)\), \(\klin_1(n, A)\) (when it is a group), \(\klin_2(n, A)\), and \(\diag(n, A)\).

Consider the root system \(\Phi\) of type \(A_{n - 1}\). It consists of so-called roots \(\mathrm e_j - \mathrm e_i\) in the vector space \(\mathbb R^n\), where \(i, j \in \{1, \ldots, n\}\) are distinct elements. Usually \(\Phi\) is considered as a subset of its linear span \(\mathbb R\Phi\) of dimension \(n - 1\). For every root \(\alpha = \mathrm e_j - \mathrm e_i\) we associate the root subgroup \(x_\alpha(*) = x_{ij}(*)\). Then the Steinberg relations imply that
\[[x_\alpha(*), x_\beta(*)] \leq \prod_{\substack{p\alpha + q\beta \in \Phi\\ p, q > 0}} x_{p\alpha + q\beta}(*)\]
for all non-antiparallel roots \(\alpha\) and \(\beta\). Here the right hand side is a nilpotent subgroup of the Steinberg group (in our case even abelian). The automorphism group of the root system \(\Phi\) is the group \(\mathrm S_n \times \mathbb Z / 2 \mathbb Z\), it acts on \(\Phi\) by permutations of coordinates and by \(\alpha \mapsto -\alpha\). This group also acts by permutations on the family of idempotents \(e_1, \ldots, e_n\), but not on the Steinberg group in general (since \(R\) is not a matrix ring). Note that the automorphism group of \(\Phi\) acts transitively on all roots and on all pairs of roots with fixed angle between them.

Now let \(\alpha \in \Phi\) be a root. The image of the set \(\Phi \setminus \{-\alpha, \alpha\}\) in the factor-space \(\mathbb R \Phi / \mathbb R\alpha\) is denoted by \(\Phi / \alpha\). The fibers of the map \(\Phi \setminus \{-\alpha, \alpha\} \rar \Phi / \alpha\) are precisely the so-called \(\alpha\)-series of roots. We claim that there is unique dot product on the factor-space \(\mathbb R \Phi / \mathbb R\alpha\) making \(\Phi / \alpha\) a root system of type \(A_{n - 2}\). Indeed, without loss of generality \(\alpha = \mathrm e_n - \mathrm e_{n - 1}\). Then \(\mathbb R^n / \mathbb R\alpha\) is isomorphic to \(\mathbb R^{n - 1}\) under the map \(\sum_i p_i \mathrm e_i \mapsto \sum_{i < n - 1} p_i \mathrm e_i + (p_{n - 1} + p_n) \mathrm e_{n - 1}\), so \(\Phi / \alpha\) maps to the standard root system of type \(A_{n - 2}\). The dot product is unique by general theory of root systems.

We show that for any root \(\alpha \in \Phi\) the root system \(\Phi / \alpha\) is actually a root system corresponding to the smaller family of idempotents in \(R\). Without loss of generality, \(\alpha = \mathrm e_n - \mathrm e_{n - 1}\). Denote the ordinary Steinberg group by \(\stlin(R, \Phi)\) and the ordinary diagonal group by \(\diag(R, \Phi)\). The Steinberg group associated to the family of idempotents \(e_1, \ldots, e_{n - 2}, e_\infty = e_{n - 1} + e_n\) is denoted by \(\stlin(R, \Phi / \alpha)\), and similarly to the diagonal group. There is a well-defined map \(F_\alpha \colon \stlin(R, \Phi / \alpha) \rar \stlin(R, \Phi)\) given by
\begin{itemize}
\item \(x_{ij}(a) \mapsto x_{ij}(a)\) for \(i, j < n - 1\);
\item \(x_{i \infty}(a) \mapsto x_{i, n - 1}(a e_{n - 1})\, x_{in}(ae_n)\) for \(i < n - 1\);
\item \(x_{\infty j}(a) \mapsto x_{n - 1, j}(e_{n - 1} a)\, x_{nj}(e_n a)\) for \(j < n - 1\).
\end{itemize}

In other words, every root subgroup \(x_{[\beta]}(*) \leq \stlin(R, \Phi / \alpha)\) for \(\beta \in \Phi \setminus \{-\alpha, \alpha\}\) isomorphically maps onto \(\prod_{\beta + p\alpha \in \Phi} x_{\beta + p\alpha}(*)\), where the product is taken over the \(\alpha\)-series of \(\beta\). Note that \(t_\alpha(*) = \stmap(x_\alpha(*))\) lies in the new diagonal group \(\diag(R, \Phi / \alpha)\), i.e. it nicely acts on \(\stlin(R, \Phi / \alpha)\). When we need an iterated factor \((\Phi / \alpha) / [\beta]\) of the root system, we denote in by \(\Phi / \{\alpha, \beta\}\) (of course, it is canonically isomorphic to \(\Phi / \{\beta, \alpha\}\), they have the same corresponding families of idempotents in \(R\)). Also the root systems \(\Phi / \alpha\) and \(\Phi / (-\alpha)\) are canonically isomorphic, so we may identify \(F_\alpha\) with \(F_{-\alpha}\). Later we prove that \(F_\alpha\) is often an isomorphism.

Actually, our construction of \(\Phi / \alpha\) and a group \(\stlin(R, \Phi / \alpha)\) is a very particular case of a general notion of relative root systems from \cite{PetrovStavrova}, where the same is done for elementary subgroups of isotropic reductive groups.

\section{Pro-objects}

We recall a construction of the pro-completion of a given category \(\mathcal C\) from \cite{CategoryTextbook}, section 6.1. A small category \(\mathcal I\) is called filtered if it is non-empty, for any two objects \(i_1, i_2\) there is a diagram \(i_1 \rar i_3 \lar i_2\) in \(\mathcal I\), and every pair of parallel arrows \(i_1 \rightrightarrows i_2\) equalizes by some arrow \(i_2 \rar i_3\) in \(\mathcal I\). A pro-object in \(\mathcal C\) is a contravariant functor \(X^{(\infty)}\) from a filtered category \(\mathcal I_X\) to \(\mathcal C\). Objects of \(\mathcal I_X\) are called indices of \(X^{(\infty)}\). We write \(X^{(i)}\) for the values of \(X^{(\infty)}\) on indices \(i\) and omit the values of \(X^{(\infty)}\) on arrows \(i \rar j\) in our formulas if the arrow \(i \rar j\) is clear from the context (for example, if \(j\) is a sufficiently large index). So a pro-object \(X^{(\infty)}\) is the formal projective limit of \(X^{(i)}\). We use the notation with upper indices since our pro-objects consist of homotopes of various algebraic objects.

The category of pro-objects is denoted by \(\Pro(\mathcal C)\). By definition,
\[
\Pro(\mathcal C)(X^{(\infty)}, Y^{(\infty)}) = \varprojlim_{j \in \mathcal I_Y} \varinjlim_{i \in \mathcal I_X} \mathcal C(X^{(i)}, Y^{(j)}).
\]
There is a more explicit description of morphisms. We say that a pre-morphism \(f \colon X^{(\infty)} \rar Y^{(\infty)}\) consists of a function \(f^*\) from the set of indices of \(Y\) to the set of indices of \(X\), and of arrows \(f^{(i)} \colon X^{(f^*(i))} \rar Y^{(i)}\) for all \(i \in \mathcal I_Y\) such that for any arrow \(i \rar j\) in \(\mathcal I_Y\) there exists a sufficiently large index \(k \in \mathcal I_X\) making the composition \(X^{(k)} \rar X^{(f^*(i))} \rar Y^{(i)}\) equal to \(X^{(k)} \rar X^{(f^*(j))} \rar Y^{(j)} \rar Y^{(i)}\) (here \(f^*\) is not a functor between index categories). A composition of pre-morphisms \(f \colon X^{(\infty)} \rar Y^{(\infty)}\) and \(g \colon Y^{(\infty)} \rar Z^{(\infty)}\) is the pre-morphism \(g \circ f\), where \((g \circ f)^*(i) = f^*(g^*(i))\) and \((g \circ f)^{(i)} = g^{(i)} \circ f^{(g^*(i))}\). Two parallel pre-morphisms \(f, g \colon X^{(\infty)} \rar Y^{(\infty)}\) are called equivalent if for every \(i \in \mathcal I_Y\) there exists a sufficiently large index \(j \in \mathcal I_X\) making the composition \(X^{(j)} \rar X^{(f^*(i))} \rar Y^{(i)}\) equal to \(X^{(j)} \rar X^{(g^*(i))} \rar Y^{(i)}\). Finally, a morphism \(X^{(\infty)} \rar Y^{(\infty)}\) is an equivalence class of pre-morphisms. Note that equivalence is preserved under compositions.

The category \(\mathcal C\) embeds into \(\Pro(\mathcal C)\) (i.e. there is a fully faithful functor between them), because we may consider every object from \(\mathcal C\) as a pro-object with a single index and a single arrow between indices. So \(\Pro(\mathcal C)(X, Y) \cong \mathcal C(X, Y)\), \(\Pro(\mathcal C)(X^{(\infty)}, Y) \cong \varinjlim_{i \in \mathcal I_X} \mathcal C(X^{(i)}, Y)\), and \(\Pro(\mathcal C)(X, Y^{(\infty)}) \cong \varprojlim_{i \in \mathcal I_Y} \mathcal C(X, Y^{(i)})\). Also \(X^{(\infty)}\) is the projective limit of \(X^{(i)}\) in the category \(\Pro(\mathcal C)\).

The category of pro-sets \(\Pro(\Set)\) has all finite limits by \cite{CategoryTextbook}, proposition 6.1.18. Hence we may consider algebraic objects in \(\Pro(\Set)\) such as rings and groups. Any algebraic formula (say, the commutator or a polynomial with integer coefficients) defines a morphism in \(\Pro(\Set)\) from a product of algebraic objects to an algebraic object. If \(a\) is a variable in such a formula, then \(a \in X^{(\infty)}\) means that \(X^{(\infty)}\) is the domain of \(a\). If \(X^{(\infty)}\) and \(Y^{(\infty)}\) are pro-sets with the same index category \(\mathcal I\), then we may construct their product \(Z^{(\infty)}\) as follows. The index category of \(Z^{(\infty)}\) is \(\mathcal I\), \(Z^{(i)} = X^{(i)} \times Y^{(i)}\), the projection \(Z^{(\infty)} \rar X^{(\infty)}\) is given by the pre-morphism \(\pi_X^*(i) = i\), \(\pi_X^{(i)}(x, y) = x\), and similarly for the projection \(Z^{(\infty)} \rar Y^{(\infty)}\). Also the diagonal morphism \(X^{(\infty)} \rar X^{(\infty)} \times X^{(\infty)}\) is given by the pre-morphism \(\Delta^*(i) = i\), \(\Delta^{(i)}(x) = (x, x)\).

The category of pro-groups \(\Pro(\Group)\) embeds into \(\Pro(\Set)\). Every pro-group is a group object in \(\Pro(\Set)\). It is easy to see that a morphism \(f \in \Pro(\Set)(G^{(\infty)}, H^{(\infty)})\) between pro-groups comes from \(\Pro(\Group)\) if and only if it is a morphism of group objects.

\begin{lemma}\label{ProEpimorphism}
A morphism \(f \colon X^{(\infty)} \rar Y^{(\infty)}\) of pro-sets is an epimorphism if and only if the map \(\varinjlim_{j \in \mathcal I_Y} \Set(Y^{(j)}, T) \rar \varinjlim_{i \in \mathcal I_X} \Set(X^{(i)}, T)\) induced by \(f\) is injective for all sets \(T\). If \(f \colon X^{(\infty)} \rar Y^{(\infty)}\) is an epimorphism in \(\Pro(\Set)\), then \(f \times Z^{(\infty)} \colon X^{(\infty)} \times Z^{(\infty)} \rar Y^{(\infty)} \times Z^{(\infty)}\) is also an epimorphism.
\end{lemma}
\begin{proof}
The first claim follows from the fact that a limit in \(\Set\) of injective maps is injective. Take an epimorphism \(f \colon X^{(\infty)} \rar Y^{(\infty)}\) and a pro-set \(Z^{(\infty)}\). By \cite{CategoryTextbook}, theorem 6.4.3 we may assume all three pro-sets have the same index category \(\mathcal I\) and \(f\) is a pre-morphism with \(f^*(i) = i\). Let \(T\) be any set, \(i\) be an index, and \(u, v \colon Y^{(i)} \times Z^{(i)} \rar T\) be two maps such that there is an index \(j\) making the two compositions \(X^{(j)} \times Z^{(j)} \rar Y^{(j)} \times Z^{(j)} \rar T\) equal. It follows that the compositions \(X^{(j)} \rar Y^{(j)} \rar \Set(Z^{(j)}, T)\) are also equal. Since \(f\) is an epimorphism, for sufficiently large index \(k\) the two maps \(Y^{(k)} \rar \Set(Z^{(j)}, T)\) coincide, i.e. the maps \(Y^{(k)} \times Z^{(k)} \rar T\) coincide.
\end{proof}

Now let us return to the Steinberg groups. From now on suppose that \(R\) is a \(K\)-algebra for some commutative unital ring \(K\) and fix a multiplicative subset \(S \leq K^\bullet\). Then there is the localization map \(\Psi \colon R \rar S^{-1} R\), where \(S^{-1} R\) is an \(S^{-1} K\)-algebra with a complete orthogonal family of Morita equivalent idempotents. We construct a filtered category \(\mathcal S\). Its objects are the elements of \(S\), its morphisms \(s \rar s'\) are all \(s'' \in S\) such that \(ss'' = s'\), composition and the identity arrows are obvious. By default, our pro-sets have \(\mathcal S\) as the category of indices.

Let \(R^{(\infty)}\) be the pro-group with the index category \(\mathcal S\), \(R^{(s)} = \{a^{(s)} \mid a \in R\}\), and the group operation \(a^{(s)} + b^{(s)} = (a + b)^{(s)}\). The structure maps are \(a^{(ss')} \mapsto (s'a)^{(s)}\). There is a pre-morphism \(R^{(\infty)} \times R^{(\infty)} \rar R^{(\infty)}\) given by \((a^{(s)}, b^{(s)}) \mapsto (sab)^{(s)}\). Hence \(R^{(\infty)}\) becomes a non-unital ring object in \(\Pro(\Set)\), its components \(R^{(s)}\) are the homotopes of \(R\). Similarly, \(R_{ij}^{(\infty)}\) are pro-groups with the index category \(\mathcal S\), where \(R_{ij}^{(s)} = \{a^{(s)} \mid a \in R_{ij}\}\). There are also multiplication pre-morphisms \(R_{ij}^{(\infty)} \times R_{jk}^{(\infty)} \rar R_{ik}^{(\infty)}, (a, b) \mapsto ab\). We usually write the elements of \(R_{ij}^{(s)} \times R_{jk}^{(s)}\) as \(a^{(s)} \otimes b^{(s)}\), variables with the domain \(R_{ij}^{(\infty)} \times R_{jk}^{(\infty)}\) as \(a \otimes b\), and the multiplication morphisms as \(a \otimes b \mapsto ab\). The following lemmas essentially say that \(R_{ik}^{(\infty)} \cong R_{ij}^{(\infty)} \otimes_{R_{jj}^{(\infty)}} R_{jk}^{(\infty)}\) as pro-groups, though we do not define tensor products of abelian pro-groups.

\begin{lemma}\label{RingGeneration}
Let \(i, j, k \in \{1, \ldots, n\}\) be indices. Then there is a positive integer \(N\) such that the morphism 
\[m_N \colon \prod_{p = 1}^N \bigl(R_{ij}^{(\infty)} \times R_{jk}^{(\infty)}\bigr) \rar R_{ik}^{(\infty)}, \enskip (a_p \otimes b_p)_{p = 1}^N \mapsto \sum_{p = 1}^N a_p b_p\]
is an epimorphism of pro-sets.
\end{lemma}
\begin{proof}
Let \(e_i = \sum_{p = 1}^N x_p y_p\) for \(x_p \in R_{ij}\) and \(y_p \in R_{ji}\), such a decomposition exists by Morita equivalence of \(e_i\) and \(e_j\). Take a set \(T\) and two maps \(f, g \colon R_{ik}^{(s)} \rar T\) such that
\[f\bigl(\sum_{p = 1}^N a_p^{(ss')} b_p^{(ss')}\bigr) = g\bigl(\sum_{p = 1}^N a_p^{(ss')} b_p^{(ss')}\bigr)\]
for all \(a_p \in R_{ij}\) and \(b_p \in R_{jk}\). Then \(f\) and \(g\) coincide on \(R_{ik}^{(s^2 {s'}^2)}\), because
\[(ss' c)^{(ss')} = \sum_{p = 1}^N x_p^{(ss')} \otimes (y_p c)^{(ss')}\]
for all \(c \in R_{ik}\). By lemma \ref{ProEpimorphism}, \(m_N\) is an epimorphism.
\end{proof}

\begin{lemma}\label{RingPresentation}
Let \(i, j, k \in \{1, \ldots, n\}\) be indices and \(G^{(\infty)}\) be a pro-group. Then a morphism of pro-sets \(g \colon R_{ij}^{(\infty)} \times R_{jk}^{(\infty)} \rar G^{(\infty)}\) factors as \(f \circ m_{ijk}\) for some morphism \(f \colon R_{jk}^{(\infty)} \rar G^{(\infty)}\) of pro-groups if and only if \(g\) satisfies the identities
\begin{itemize}
\item \([g(a \otimes b), g(a' \otimes b')] = 1\);
\item \(g((a + a') \otimes b) = g(a \otimes b)\, g(a' \otimes b)\);
\item \(g(a \otimes (b + b')) = g(a \otimes b)\, g(a \otimes b')\);
\item \(g(ar \otimes b) = g(a \otimes rb)\) for \(a \in R_{ij}^{(\infty)}\), \(r \in R_{jj}^{(\infty)}\), \(b \in R_{jk}^{(\infty)}\).
\end{itemize}
\end{lemma}
\begin{proof}
The necessity of the identities is clear. By lemma \ref{RingGeneration}, it suffices to consider a group \(G\) instead of a pro-group. Let \(e_i = \sum_p x_p y_p\) for \(x_p \in R_{ij}\) and \(y_p \in R_{ji}\). Take a morphism \(g\) satisfying the identities. For sufficiently large \(s \in S\) the morphism \(g\) is given by a map \(g' \colon R_{ij}^{(s)} \times R_{jk}^{(s)} \rar G\) satisfying the first three identities and the identity \(g'\bigl((sar)^{(s)} \otimes b^{(s)}\bigr) = g'\bigl(a^{(s)} \otimes (srb)^{(s)}\bigr)\) for all \(a \in R_{ij}\), \(r \in R_{jj}\), \(b \in R_{jk}\). Consider the homomorphism \(f' \colon R_{ik}^{(s^2)} \rar G\) given by 
\[f'\bigl(c^{(s^2)}\bigr) = \prod_p g'\bigl(x_p^{(s)} \otimes (y_p c)^{(s)}\bigr).\]
Then for all \(a \in R_{ij}\) and \(b \in R_{jk}\) we have
\begin{align*}
f'\bigl(a^{(s^2)} b^{(s^2)})
&= \prod_p g' \bigl( x^{(s)}_p \otimes (s^2 y_p ab)^{(s)} \bigr)\\
&= \prod_p g' \bigl( (sx_py_p a)^{(s)} \otimes (sb)^{(s)} \bigr)\\
&= g'\bigl(a^{(s^2)} \otimes b^{(s^2)}\bigr).
\end{align*}

It is clear that \(f'\) gives the required morphism of pro-groups.
\end{proof}

Note that \(R^{(\infty)} \cong \prod_{i, j = 1}^n R_{ij}^{(\infty)}\) as a pro-group. If \(e_1, \ldots, e_{n - 2}, e_\infty = e_{n - 1} + e_n\) is another family of idempotents, then \(R_{i\infty}^{(\infty)} \cong R_{i, n - 1}^{(\infty)} \times R_{in}^{(\infty)}\) for \(i < n - 1\), \(R_{\infty j}^{(\infty)} \cong R_{n - 1, j}^{(\infty)} \times R_{nj}^{(\infty)}\) for \(j < n - 1\), and \(R_{\infty\infty}^{(\infty)} \cong R_{n - 1, n - 1}^{(\infty)} \times R_{n - 1, n}^{(\infty)} \times R_{n, n - 1}^{(\infty)} \times R_{nn}^{(\infty)}\).

\section{Homotopes of Steinberg groups}

In this section we construct a certain pro-group \(\stlin(R)^{(\infty)}\) and prove its basic properties. The ``homotopes'' of \(\stlin(R)\) are the groups \(\stlin(R)^{(s)}\) parameterized by \(s \in S\). They are generated by symbols \(x_{ij}^{(s)}(a)\) for \(i \neq j\), \(a \in R_{ij}\) with the modified Steinberg relations
\begin{enumerate}[label = (St\arabic*$^{(s)}$)]
\item \(x_{ij}^{(s)}(a)\, x_{ij}^{(s)}(b) = x_{ij}^{(s)}(a + b)\);
\item \([x_{ij}^{(s)}(a), x_{kl}^{(s)}(b)] = 1\) for \(j \neq k\) and \(i \neq l\);
\item \([x_{ij}^{(s)}(a), x_{jk}^{(s)}(b)] = x_{ik}^{(s)}(sab)\) for \(i \neq k\).
\end{enumerate}
The structure homomorphisms are given by \(x^{(ss')}(a) \mapsto x^{(s)}(s'a)\), and a pro-group \(\stlin(R)^{(\infty)}\) is the formal projective limit of \(\stlin(R)^{(s)}\). There are pre-morphisms \(x_{ij} \colon R_{ij}^{(\infty)} \rar \stlin(R)^{(\infty)}\) of pro-groups with \(a^{(s)} \mapsto x_{ij}^{(s)}(a)\). There is also a pre-morphism \(\mathrm{st} \colon \stlin(R)^{(\infty)} \rar \glin(R)^{(\infty)}\) of pro-groups, where \(\glin(R)^{(s)}\) is the group of quasi-invertible elements of \(R^{(s)}\) (so \(\glin(R)^{(1)}\) is only isomorphic to \(\glin(R)\)) and \(\mathrm{st}^{(s)} \colon \stlin(R)^{(s)} \rar \glin(R)^{(s)}, x_{ij}^{(s)}(a) \mapsto a^{(s)}\). It follows that all maps \(x_{ij}^{(s)} \colon R_{ij} \rar \stlin(R)^{(s)}\) are injective.

\begin{lemma}\label{SteinbergPresentation}
Let \(G^{(\infty)}\) be a pro-group. Then every morphism \(f \colon \stlin(R)^{(\infty)} \rar G^{(\infty)}\) of pro-groups is uniquely determined by its compositions with all \(x_{ij} \colon R_{ij}^{(\infty)} \rar \stlin(R)^{(\infty)}\). Morphisms \(g_{ij} \colon R_{ij}^{(\infty)} \rar G^{(\infty)}\) of pro-groups are obtained in this way if and only if they satisfy (St1)--(St3) in \(\Pro(\Set)\).
\end{lemma}
\begin{proof}
This follows directly from the definitions. Here we need that \(G^{(\infty)}\) is actually a pro-group, not just a group object in \(\Pro(\Set)\).
\end{proof}

Recall that \(\Phi\) is the root system of type \(A_{n - 1}\), it parametrizes generators \(x_\alpha\) of the pro-group \(\stlin(R, \Phi)^{(\infty)} = \stlin(R)^{(\infty)}\). For every root \(\alpha \in \Phi\) the pro-group \(\stlin(R, \Phi / \alpha)^{(\infty)}\) is defined using the smaller family of idempotents. Now \(F_\alpha \colon \stlin(R, \Phi / \alpha)^{(\infty)} \rar \stlin(R, \Phi)^{(\infty)}\) is a pre-morphism of pro-groups (with \(F_\alpha^*(s) = s\)). We are ready to prove that \(F_\alpha\) is an epimorphism for \(n \geq 3\) and an isomorphism for \(n \geq 4\) in \(\Pro(\Group)\). If \(S = \{1\}\), then this means that \(F_\alpha \colon \stlin(R, \Phi / \alpha) \rar \stlin(R, \Phi)\) is a surjection for \(n \geq 3\) and a bijection for \(n \geq 4\). The proof also shows that the group \(\stlin(R, \Phi)\) is perfect for \(n \geq 3\).

\begin{lemma}\label{FactorRootGeneration}
If \(n \geq 3\) and \(\alpha \in \Phi\) is a root, then \(F_\alpha\) is an epimorphism in \(\Pro(\Group)\).
\end{lemma}
\begin{proof}
We have to show that two pre-morphisms \(f_1, f_2 \colon \stlin(R, \Phi)^{(\infty)} \rar G\) are equivalent if their compositions with \(F_\alpha\) are equivalent. By lemma \ref{SteinbergPresentation} it suffices to show that \(f_i \circ x_\beta\) is equivalent to \(f_2 \circ x_\beta\) for all \(\beta \in \Phi\). This is clear if \(\beta \neq \pm \alpha\) by definition of \(F_\alpha\). Since we may change the sign of \(\alpha\) and apply an automorphism of \(\Phi\), it remains to check the case \(\beta = \alpha = \mathrm e_n - \mathrm e_{n - 1}\). Note that the index \(1\) is different from \(n\) and \(n - 1\). We have
\[x_{n - 1, n}(ab) = [x_{n - 1, 1}(a), x_{1 n}(b)] = F_\alpha([x_{\infty 1}(a), x_{1 \infty}(b)]),\]
where \(e_\infty = e_{n - 1} + e_n\) is the new idempotent, \(a \in R_{n - 1, 1}^{(\infty)}\), \(b \in R_{1n}^{(\infty)}\). The claim now follows by lemma \ref{RingGeneration}.
\end{proof}

\begin{prop}\label{FactorRootPresentation}
If \(n \geq 4\) and \(\alpha \in \Phi\) is a root, then \(F_\alpha\) is an isomorphism in \(\Pro(\Group)\).
\end{prop}
\begin{proof}
Firsly, we construct morphisms \(\widetilde x_{ij} \colon R_{ij}^{(\infty)} \rar \stlin(R, \Phi / \alpha)^{(\infty)}\) of pro-groups and show that they satisfy the Steinberg relations. Without loss of generality, \(\alpha = \mathrm e_n - \mathrm e_{n - 1}\). We denote the generators of \(\stlin(R, \Phi)\) and \(\stlin(R, \Phi / \alpha)\) by \(x_{ij}\), in the first case \(i, j \in \{1, \ldots, n\}\), and in the second case \(i, j \in \{1, \ldots, n - 2, \infty\}\) (where \(e_\infty = e_{n - 1} + e_n\)).

Let
\begin{itemize}
\item \(\widetilde x_{ij}(a) = x_{ij}(a)\) for \(i, j < n - 1\);
\item \(\widetilde x_{ij}(a) = x_{i\infty}(a)\) for \(i < n - 1\) and \(j \in \{n - 1, n\}\);
\item \(\widetilde x_{ij}(a) = x_{\infty j}(a)\) for \(i \in \{n - 1, n\}\) and \(j < n - 1\);
\item \(\widetilde x_{n - 1, i, n}(a \otimes b) = [\widetilde x_{n - 1, i}(a), \widetilde x_{i n}(b)]\) for \(i < n - 1\), \(a \in R_{n - 1, i}^{(\infty)}\), \(b \in R_{in}^{(\infty)}\).
\end{itemize}
So we have \(\widetilde x_\beta\) for all \(\beta \neq \pm \alpha\). They clearly satisfy the Steinberg relations not involving \(\pm \alpha\). Since no Steinberg relation involves both \(\alpha\) and \(-\alpha\), it suffices to construct \(\widetilde x_\alpha\) and prove the relations with \(\alpha\). The idea is to take \(\widetilde x_\alpha(ab) = \widetilde x_{n - 1, i, n}(a \otimes b)\) for \(a \in R_{n - 1, i}^{(\infty)}\) and \(b \in R_{in}^{(\infty)}\), and show that this is well-defined using lemma \ref{RingPresentation}. Below we use already proved Steinberg relations, the commutator identities, and the Hall -- Witt identity without further mention.

If \(j \neq i, n\) and \(k \neq i, n - 1\), then
\begin{align*}
\up{\widetilde x_{jk}(a)}{\widetilde x_{n - 1, i, n}(b \otimes c)} &= \up{\widetilde x_{jk}(a)}{[\widetilde x_{n - 1, i}(b), \widetilde x_{in}(c)]}\\
&= [\up{\widetilde x_{jk}(a)}{\widetilde x_{n - 1, i}(b)}, \up{\widetilde x_{jk}(a)}{\widetilde x_{in}(c)}]\\
&= [\widetilde x_{n - 1, i}(b), \widetilde x_{in}(c)] = \widetilde x_{n - 1, i, n}(b \otimes c),
\end{align*}
i.e. \(\widetilde x_{jk}\) commutes with \(\widetilde x_{n - 1, i, n}\). The generators \(\widetilde x_{ji}\) for \(j \neq n - 1, n\) also commute with \(\widetilde x_{n - 1, i, n}\) because
\begin{align*}
\up{\widetilde x_{ji}(a)}{\widetilde x_{n - 1, i, n}(b \otimes c)} &= [\widetilde x_{n - 1, i}(b), \widetilde x_{in}(c)\, \widetilde x_{jn}(\mu_{jin}(a \otimes c))]\\
&= \widetilde x_{n - 1, i, n}(b \otimes c).
\end{align*}
It follows that \(\widetilde x_{n - 1, i}(ab)\) commutes with \(\widetilde x_{n - 1, i, n}\) for \(a \in R_{n - 1, j}^{(\infty)}\) and \(b \in R_{jn}^{(\infty)}\) if \(j\) is distinct from \(i\), \(n - 1\), \(n\) (such a \(j\) exists since \(n \geq 4\)). By lemmas \ref{ProEpimorphism} and \ref{RingGeneration}, \(\widetilde x_{n - 1, i, n}\) commutes with \(\widetilde x_{n - 1, i}\). It may be proved similarly (or using an outer automorphism of \(\Phi\) preserving \(\alpha\)) that \(\widetilde x_{n - 1, i, n}\) commutes with \(\widetilde x_{ik}\) for all \(k \neq n - 1\). Hence we have all cases of (St2).

Now let us show that \(\widetilde x_{n - 1, i, n}(a \otimes b)\) is biadditive. Indeed,
\begin{align*}
\widetilde x_{n - 1, i, n}((a + a') \otimes b) &= [\widetilde x_{n - 1, i}(a + a'), \widetilde x_{in}(b)]\\
&= [\widetilde x_{n - 1, i}(a), \widetilde x_{in}(b)]\,
[\widetilde x_{n - 1, i}(a'), \widetilde x_{in}(b)]\\
&= \widetilde x_{n - 1, i, n}(ab)\, \widetilde x_{n - 1, i, n}(a'b),
\end{align*}
and similarly for the second variable.

If \(i \neq j\) are indices from \(\{1, \ldots, n - 2\}\), then for \(a \in R_{n - 1, i}^{(\infty)}\), \(b \in R_{ij}^{(\infty)}\), \(c \in R_{jn}^{(\infty)}\) we have
\begin{align*}
\widetilde x_{n - 1, i, n}(a \otimes bc) &= [\widetilde x_{n - 1, i}(a), [\widetilde x_{ij}(b), \widetilde x_{jn}(c)]]\\
&= [[\widetilde x_{n - 1, i}(a), \widetilde x_{ij}(b)], \up{\widetilde x_{ij}(b)}{\widetilde x_{jn}(c)}]\\
&= [\widetilde x_{n - 1, j}(ab), \widetilde x_{jn}(c)\, \widetilde x_{in}(bc)]\\
&= \widetilde x_{n - 1, j, n}(ab \otimes c).
\end{align*}

It follows that \(\widetilde x_{n - 1, i, n}(a \otimes bcd) = \widetilde x_{n - 1, i, n}(abc \otimes d)\) for \(a \in R_{n - 1, i}^{(\infty)}\), \(b \in R_{ij}^{(\infty)}\), \(c \in R_{ji}^{(\infty)}\), \(d \in R_{in}^{(\infty)}\). By lemmas \ref{ProEpimorphism} and \ref{RingGeneration}, the last identity from the statement of lemma \ref{RingPresentation} holds, so the morphism \(\widetilde x_{n - 1, n} \colon R_{n - 1, n}^{(\infty)} \rar \stlin(R, \Phi / \alpha)\) is well-defined by \(\widetilde x_{n - 1, n}(ab) = \widetilde x_{n - 1, 1, n}(a \otimes b)\) for \(a \in R_{n - 1, 1}^{(\infty)}\) and \(b \in R_{1n}^{(\infty)}\), it satisfies (St1). By the relation between \(\widetilde x_{n - 1, i, n}\) and \(\widetilde x_{n - 1, j, n}\) we have (St3) with \(\alpha\) in the right hand side.

It remains to prove (St3) with \(\alpha\) in the left hand side. If \(i\) and \(j\) are different indices less that \(n - 1\), then
\begin{align*}
[\widetilde x_{i, n - 1}(a), \widetilde x_{n - 1, j, n}(b \otimes c)] &= [\widetilde x_{i, n - 1}(a), [\widetilde x_{n - 1, j}(b), \widetilde x_{jn}(c)]]\\
&= [[\widetilde x_{i, n - 1}(a), \widetilde x_{n - 1, j}(b)], \widetilde x_{jn}(c)\, \widetilde x_{n - 1, j, n}(b \otimes c)]\\
&= [\widetilde x_{ij}(ab), \widetilde x_{jn}(c)\, \widetilde x_{n - 1, j, n}(b \otimes c)] = \widetilde x_{in}(abc),
\end{align*}
hence \([\widetilde x_{i, n - 1}(a), \widetilde x_{n - 1, n}(b)] = \widetilde x_{in}(ab)\) by lemmas \ref{ProEpimorphism}, \ref{RingGeneration}, and by commutativity of \(\widetilde x_{n - 1, n}\) with \(\widetilde x_{in}\) (here we also use that \(n \geq 4\)). The last relation \([\widetilde x_{n - 1, n}(a), \widetilde x_{ni}(b)] = \widetilde x_{n - 1, i}(ab)\) may be shown similarly or again using an outer automorphism.

Now we have morphism \(G_\alpha \colon \stlin(R, \Phi)^{(\infty)} \rar \stlin(R, \Phi / \alpha)^{(\infty)}\) of pro-groups by lemma \ref{SteinbergPresentation}. It is clear from the construction that \(G_\alpha \circ F_\alpha\) is the identity. Hence \(F_\alpha \circ G_\alpha\) is also the identity by lemma \ref{FactorRootGeneration}.
\end{proof}

\section{Globalization}

First we show that the the semi-direct product \(\stlin(S^{-1} R) \rtimes \diag(S^{-1} R)\) acts on the pro-group \(\stlin(R)^{(\infty)}\). The next lemma also shows that \(\glin(S^{-1} R)^*\) acts on the pro-group \(\glin(R)^{(\infty)}\) (and on the ring pro-object \(R^{(\infty)}\)) by inner automorphisms.

\begin{lemma}\label{RingLocalAction}
Let \(i, j \in \{1, \ldots, n\}\) be indices. Then for every \(\frac as \in S^{-1} R_{ii}\) the family of maps 
\[\mathrm L_{(a, s)}^{(s')} \colon R_{ij}^{(ss')} \rar R_{ij}^{(s')},\enskip b^{(ss')} \mapsto (ab)^{(s')}\]
gives a well-defined endomorphism \(\mathrm L_{\frac as}\) of the pro-group \(R_{ij}^{(\infty)}\). Moreover, \(\mathrm L\) is a ring homomorphism from \(S^{-1} R_{ii}\) to the ring \(\Pro(\Group)(R_{ij}^{(\infty)})\). There is a similarly defined ring anti-homomorphism \(\mathrm R\) from \(S^{-1} R_{jj}\) to \(\Pro(\Group)(R_{ij}^{(\infty)})\). They satisfy the identities
\[
\mathrm L_u(ab) = \mathrm L_u(a) b,\quad
\mathrm R_v(a) b = a \mathrm L_v(b),\quad
\mathrm R_w(ab) = a \mathrm R_w(b)
\]
in \(\Pro(\Set)\) for all \(u \in S^{-1} R_{ii}\), \(v \in S^{-1} R_{jj}\), \(w \in S^{-1} R_{kk}\) and for \(a \in R_{ij}^{(\infty)}\), \(b \in R_{jk}^{(\infty)}\).
\end{lemma}
\begin{proof}
For any \(a\) and \(s\) this family is a pre-endomorphism \(\mathrm L_{(a, s)}\) of \(R_{ij}^{(\infty)}\). We show that \(\mathrm L_{(a, s)}\) is equivalent to \(\mathrm L_{(as', ss')}\) for all \(s' \in S\). Indeed, if \(s'' \in S\) is an index, then
\[\mathrm L_{(a, s)}^{(s'')}\bigl(b^{(ss's'')}\bigr) = (s' ab)^{(s'')} = \mathrm L_{(as', ss')}^{(s'')}\bigl(b^{(ss's'')}\bigr)\]
Clearly, \(\mathrm L\) is then a homomorphism of rings, and similarly \(\mathrm R\) is an anti-homomorphism of rings. The first identity follows from
\[\mathrm L_{(a, s)}^{(s')}\bigl(b^{(ss')} c^{(ss')}\bigr) = (ss'abc)^{(s')} = \mathrm L^{(s')}_{(a, s)}\bigl(b^{(ss')}\bigr)\, c^{(ss')},\]
and the other two may be proved similarly.
\end{proof}

\begin{prop}\label{SteinbergLocalAction}
If \(n \geq 4\), then there is an action \(\Ad\) of \(\stlin(S^{-1} R) \rtimes \diag(S^{-1} R)\) on \(\stlin(R)^{(\infty)}\) by automorphisms such that \(\mathrm{st} \colon \stlin(R)^{(\infty)} \rar \glin(R)^{(\infty)}\) is equivariant. It is given by
\begin{itemize}
\item \(\Ad_{d_i(u)}(x_{jk}(a)) = x_{jk}(a)\) for \(k \neq i \neq j\);
\item \(\Ad_{d_i(u)}(x_{ij}(a)) = x_{ij}(\mathrm L_u(a))\);
\item \(\Ad_{d_i(u)}(x_{ji}(a)) = x_{ji}(\mathrm R_{u^{-1}}(a))\);
\item \(\Ad_{x_{ij}(v)}(x_{kl}(a)) = x_{kl}(a)\) for \(j \neq k\) and \(l \neq i\);
\item \(\Ad_{x_{ij}(v)}(x_{jk}(a)) = x_{ik}(\mathrm L_v(a))\, x_{jk}(a)\) for \(i \neq k\);
\item \(\Ad_{x_{ij}(v)}(x_{ki}(a)) = x_{kj}(\mathrm R_v(a))\, x_{ki}(a)\) for \(j \neq k\).
\end{itemize}
\end{prop}
\begin{proof}
For every \(u \in S^{-1} R\) the first three formulas give a well-defined endomorphism \(\Ad_u\) of the pro-group \(\stlin(R)^{(\infty)}\) by lemmas \ref{SteinbergPresentation} and \ref{RingLocalAction}. The properties \(\Ad_{d_i(ab)} = \Ad_{d_i(a)} \circ \Ad_{d_i(b)}\) and \(\Ad_{d_i(a)} \circ \Ad_{d_j(c)} = \Ad_{d_j(c)} \circ \Ad_{d_i(a)}\) for all distinct \(i, j\) follow from the same lemmas. Now note that \(\diag(S^{-1} R)\) is generated by all \((S^{-1} R_{ii})^*\) with the relations \(d_i(ab) = d_i(a) d_i(b)\) and \(d_i(a) d_j(c) = d_j(c) d_i(a)\) as an abstract group. Here we do not need the condition \(n \geq 4\).

The automorphism \(\Ad_{x_\alpha(v)}\) is well-defined on \(\stlin(R, \Phi / \alpha)^{(\infty)}\), because then the element \(x_\alpha(v)\) may be considered as its image in \(\diag(S^{-1} R, \Phi / \alpha)\). Hence by proposition \ref{FactorRootPresentation} this is also an automorphism of \(\stlin(R, \Phi)^{(\infty)}\). We still have to show that \(\Ad_{x_\alpha(v)}\) stabilizes the generator \(x_\alpha\) of \(\stlin(R, \Phi)^{(\infty)}\). Without loss of generality, let \(\alpha = \mathrm e_{n - 1} - \mathrm e_n\). Other relations imply that \(\Ad_{x_{n - 1, n}(v)}\) stabilizes \(x_{n - 1, n}(ab)\) for \(a \in R_{n - 1, 1}^{(\infty)}\) and \(b \in R_{1n}^{(\infty)}\). Hence it stabilizes \(x_{n - 1, n}(c)\) by lemma \ref{RingGeneration}.

By lemma \ref{RingLocalAction}, the relations between \(d_i(u)\) and \(x_{jk}(v)\) hold for these automorphisms, as well as (St1) for \(\Ad_{x_{jk}(v)}\). It remains to prove that \(\Ad_{x_{ij}(v)}\) satisfy the Steinberg relations (St2) and (St3). If \(\alpha\) and \(\beta\) are non-parallel roots, then the relation between \(\Ad_{x_\alpha(v)}\) and \(\Ad_{x_\beta(v')}\) holds on the pro-group \(\stlin(R, \Phi / \{\alpha, \beta\})^{(\infty)}\) (since it holds for the images of \(x_\alpha(v)\) and \(x_\beta(v')\) in \(\diag(S^{-1} R, \Phi / \{\alpha, \beta\})\)), hence also on the whole pro-group \(\stlin(R, \Phi)^{(\infty)}\) by lemma \ref{FactorRootGeneration}. It is clear from the definition that \(\mathrm{st} \colon \stlin(R)^{(\infty)} \rar \glin(R)^{(\infty)}\) is equivariant.
\end{proof}

For the next lemma we need maximal unipotent subgroups of the Steinberg group. Let at the moment \(R\) be arbitrary unital ring with a complete family of Morita equivalent orthogonal idempotents. The standard maximal unipotent subgroups of \(\stlin(R)\) are \(U^+ = U^+(\Phi) = \langle x_{ij}(a) \mid i < j \rangle\) and \(U^- = U^-(\Phi) = \langle x_{ij}(a) \mid i > j \rangle\). The Steinberg relations imply that \(\stmap\) maps \(U^+\) and \(U^-\) isomorphically onto the groups of upper and lower unitriangular elements of \(\glin(R)\). These groups are nilpotent. If \(\alpha = \mathrm e_{i + 1} - \mathrm e_i\) is a simple root, then the subgroups \(U^+(\Phi / \alpha)\) and \(U^-(\Phi / \alpha)\) are well-defined and we have decompositions
\[
U^+(\Phi) = U^+(\Phi / \alpha) \rtimes x_\alpha(*), \quad
U^-(\Phi) = U^-(\Phi / \alpha) \rtimes x_{-\alpha}(*).
\]
Moreover, both \(U^+(\Phi / \alpha)\) and \(U^-(\Phi / \alpha)\) are normalized by \(x_\alpha(*)\) and \(x_{-\alpha}(*)\).

\begin{lemma}\label{GaussDecomposition}
Let \(R\) be a semi-local ring with a complete family of Morita equivalent orthogonal idempotents \(e_1, \ldots, e_n\). Then there is Gauss decomposition \(\glin(R) = \stmap(U^+) \stmap(U^-) \stmap(U^+) \diag(R)\), and \(\glin(R)\) is the factor of \(\stlin(R) \rtimes \diag(R)\) by the relations
\[\prod_{k = 1}^N (x_{i, i + 1}(a_k)\, x_{i + 1, i}(b_k)) = d_i(u)\, d_{i + 1}(v)\]
whenever the images of both sides coincide in \(\glin(R)\) (it suffices to take \(N = 3\) for all \(i \in \{1, \ldots, n - 1\}\)).
\end{lemma}
\begin{proof}
Let us prove Gauss decomposition for \(n = 2\) (if \(n = 1\), there is nothing to prove). Take \(g \in \glin(R)\). We claim that there is \(a \in R_{12}\) such that \(e_2 g\, t_{12}(a)\, e_2 \in (R_{22})^*\). Without loss of generality, we may factor \(R\) by its Jacobson radical and consider only one of the simple factors, i.e. we may assume that \(R\) is a matrix ring over a division ring of size \(m\). Moreover, we may assume that \(e_1 = e'_{11} + \ldots + e'_{tt}\) and \(e_2 = e'_{t + 1, t + 1} + \ldots + e'_{mm}\), where \(e'_{ij}\) are the matrix units. Now the matrix \(e_2 g\) has rank \(m - t\), so we may add first \(t\) columns of \(e_2 g\) to the last \(m - t\) columns with some coefficients making the rank of \(e_2 g\, t_{12}(a)\, e_2\) equal to \(m - t\) (the matrix \(a\) describes the coefficients of elementary transformations). Take such an \(a\). Then there is \(b \in R_{21}\) such that \(e_2 g\, t_{12}(a)\, t_{21}(b)\, e_1 = 0\), hence \(g\, t_{12}(a)\, t_{21}(b) \in t_{12}(*) \diag(R)\). The case of arbitrary \(n\) follows by easy induction if we pass from \(\Phi\) to \(\Phi / (\mathrm e_{i + 1} - \mathrm e_i)\) and then apply the case \(n = 2\) for \((e_i + e_{i + 1}) R (e_i + e_{i + 1})\).

If \((g, h) \in \stlin(R) \rtimes \diag(R)\) maps to \(1 \in \glin(R)\) and \(g\) lies in \(U^+\, U^-\, U^+\), then it is easy to see that the factor from \(U^-\) is trivial. Hence \(g = 1\) and \(h = 1\). It remains to prove that every element from \(\stlin(R) \rtimes \diag(R)\) is congruent to an element \((g, h)\) with \(g \in U^+\, U^-\, U^+\) modulo the relations. Let \(G\) be the set of such elements. Then \(G\) is closed under multiplication by \(U^+\) and \(\diag(R)\) from the right, it also contains \(1\). It suffices to show that \(G\, x_{i + 1, i}(*) \subseteq G\) because \(\stlin(R)\) is generated by \(U^+\) and all \(x_{i + 1, i}(*)\) (this easily follows from (St3) and lemma \ref{RingGeneration}). After passing from \(\Phi\) to \(\Phi / (\mathrm e_{i + 1} - \mathrm e_i)\) it remains to show that
\[t_{i, i + 1}(*)\, t_{i + 1, i}(*)\, t_{i, i + 1}(*)\, t_{i + 1, i}(*) \subseteq t_{i, i + 1}(*)\, t_{i + 1, i}(*)\, t_{i, i + 1}(*)\, d_i(*)\, d_{i + 1}(*),\]
modulo the relations. But this follows from Gauss decomposition for \((e_i + e_{i + 1}) R (e_i + e_{i + 1})\).
\end{proof}

Recall that a sequence \(a_1, \ldots, a_n\) in a unital ring \(A\) is called left unimodular if there are \(b_1, \ldots, b_n \in A\) such that \(\sum_{i = 1}^n b_i a_i = 1\). The ring \(A\) satisfies \(\mathrm{sr}(A) \leq n - 1\) if for every left unimodular sequence \(a_1, \ldots, a_n\) there are elements \(c_1, \ldots, c_{n - 1} \in A\) such that the shorter sequence \(a_1 + c_1 a_n, \ldots, a_{n - 1} + c_{n - 1} a_n\) is also unimodular. We say that a unital \(K\)-algebra satisfies \(\mathrm{lsr}(A) \leq n - 1\) if \(\mathrm{sr}(A_{\mathfrak m}) \leq n - 1\) for every maximal ideal \(\mathfrak m \leqt K\). For example, if \(A\) is a quasi-finite \(K\)-algebra, then \(\mathrm{lsr}(A) \leq 1\).

\begin{prop}\label{LinearLocalAction}
Let \(R\) be a unital \(K\)-algebra with a complete family of Morita equivalent orthogonal idempotents \(e_1, \ldots, e_n\), \(S \leq K^\bullet\) is a multiplicative subset. Suppose that either \(n \geq 4\) or \(n = 3\) and \(S = 1\). Also suppose that either \(R = \mat(n, A)\) for an algebra \(A\) with \(\mathrm{sr}(S^{-1} A) \leq n - 2\) (where \(e_i = e_{ii}\) are the matrix units) or that \(S^{-1} R\) is semi-local. Then \(\glin(S^{-1} R)\) acts on \(\stlin(R)^{(\infty)}\), the morphism \(\stmap\) is equivariant, and this action is consistent with the action of \(\stlin(S^{-1} R) \rtimes \diag(S^{-1} R)\).
\end{prop}
\begin{proof}
Let \(\alpha_i = \mathrm e_{i + 1} - \mathrm e_i\) be the simple roots of \(\Phi\) and consider the root system \(\Phi' = \Phi / \{\alpha_2, \ldots, \alpha_{n - 2}\}\) corresponding to the family of idempotents \(e_1, e_{2'} = e_2 + \ldots + e_{n - 1}, e_n\). By propositions \ref{FactorRootPresentation} and \ref{SteinbergLocalAction} the group
\[G = \stlin(S^{-1} R, \Phi') \rtimes \diag(S^{-1} R, \Phi')\]
acts on \(\stlin(R)^{(\infty)}\) and \(\stmap\) is equivariant. Note that \(G \rar \glin(S^{-1} R)\) is surjective: in the matrix case this follows from surjective stability of \(\klin_1\) (see \cite{LinStabBass}, theorem 4.2), and in the semi-local case this follows from Gauss decomposition \ref{GaussDecomposition}. It remains to show that every generator of the kernel of \(G \rar \glin(S^{-1} R)\) acts trivially.

We claim that it is possible to choose the generators to be of type
\[\bigl(\prod_{k = 1}^N (x_{[\alpha_i]}(a_k)\, x_{[-\alpha_i]}(b_k)), h\bigr)\]
for \(i = 1\) and \(i = n - 1\). In the semi-local case this follows from lemma \ref{GaussDecomposition}. In the matrix case we may consider only the generators of type \((g, d_1(u)\, d_{2'}(v))\) and \((g, d_{2'}(u)\, d_n(v))\) for \(g \in \stlin(S^{-1} R, \Phi')\) by surjective stability of \(\klin_1\). But then \(g\) is a product of \(x_{\pm[\alpha_1]}\) or \(x_{\pm[\alpha_{n - 1}]}\) by injective stability of \(\klin_1\) (\cite{LinStabVaser}) and surjective stability of \(\klin_2\) (\cite{LinStabDennis}). Finally, the proposition follows from lemma \ref{FactorRootGeneration} applied to \(\Phi' / \alpha_1\) and \(\Phi' / \alpha_{n - 1}\).
\end{proof}

\section{Steinberg crossed module}

We are ready to prove the main results. Recall the remark before lemma \ref{FactorRootGeneration} that \(\stlin(R)\) is perfect for any unital ring \(R\) with a complete family of Morita equivalent orthogonal idempotents \(e_1, \ldots, e_n\) if \(n \geq 3\).

\begin{theorem}\label{SemilocalCrossedModule}
Let \(R\) be a semi-local unital ring with a complete family of Morita equivalent orthogonal idempotents \(e_1, \ldots, e_n\). Suppose that \(n \geq 3\). Then there is unique action of \(\glin(R)\) on \(\stlin(R)\) making \(\stmap \colon \stlin(R) \rar \glin(R)\) a crossed module, it is consistent with the action of \(\stlin(R) \rtimes \diag(R)\).
\end{theorem}
\begin{proof}
Indeed, by proposition \ref{LinearLocalAction} (applied to \(K = \mathbb Z\) and \(S = \{1\}\)) there is an action of \(\glin(R)\) on \(\stlin(R)\) such that \(\stmap\) is equivariant and this action is consistent with the conjugacy action of \(\stlin(R)\) on itself. Hence, in particular, \(\stlin(R)\) is a central perfect extension of \(\elin(R)\). By abstract group theory it follows that for every \(g \in \glin(R)\) there is at most one automorphism \(\Ad_g\) of \(\stlin(R)\) making \(\stmap\) \(g\)-equivariant.
\end{proof}

\begin{theorem}\label{QuasifiniteCrossedModule}
Let \(K\) be a commutative ring, \(R\) be a unital \(K\)-algebra with a complete family of Morita equivalent orthogonal idempotents \(e_1, \ldots, e_n\) for \(n \geq 4\). Suppose that either \(R = \mat(n, A)\) for an algebra \(A\) with \(\mathrm{lsr}(A) \leq n - 2\) (and \(e_i = e_{ii}\) are the matrix units) or \(R\) is quasi-finite. Then there is unique action of \(\glin(R)\) on \(\stlin(R)\) making \(\stmap \colon \stlin(R) \rar \glin(R)\) a crossed module, it is consistent with the action of \(\stlin(R) \rtimes \diag(R)\).
\end{theorem}
\begin{proof}
Firstly, we show any \(g \in \klin_2(R)\) lies in the center of \(\stlin(R)\). Fix two indices \(i \neq j\) and consider the ideal
\[\mathfrak a = \{k \in K \mid [g, x_{ij}(kR_{ij})] = 1\} \leqt K.\]
It suffices to show that \(\mathfrak a\) is not contained in any maximal ideal \(\mathfrak m \leqt K\). By proposition \ref{LinearLocalAction} the element \(\Psi(g)\) trivially acts on the pro-group \(\stlin(R)^{(\infty)}\), where the multiplicative subset is \(S = K \setminus \mathfrak m\) and \(\Psi\) is the localization map (the quasi-finite case easily reduces to the case of a finite \(K\)-algebra \(R\), so after localization it becomes a semi-local ring). This means that \(\mathfrak a\) contains an element from \(K \setminus \mathfrak m\).

Secondly, we show that \(\elin(R)\) is normal in \(\glin(R)\) (this result is already known at least for matrix rings). Fix \(g \in \glin(R)\) and indices \(i \neq j\). Consider the ideal
\[\mathfrak b = \{k \in K \mid \up g{t_{ij}(kR_{ij})} \in \elin(R)\} \leqt K.\]
We have to prove that \(\mathfrak b\) is not contained in any maximal ideal \(\mathfrak m \leqt K\). By proposition \ref{LinearLocalAction} the element \(\Psi(g)\) acts on the pro-group \(\stlin(R)^{(\infty)}\) such that \(\stmap\) is \(\Psi(g)\)-equivariant, where \(\Psi\) is the localization map with respect to the multiplicative subset \(S = K \setminus \mathfrak m\). Hence \(\mathfrak b\) contains an element from \(K \setminus \mathfrak m\).

Thirdly, we have to construct the action of \(\glin(R)\) on \(\stlin(R)\). Fix \(g \in \glin(R)\). There is at most one endomorphism \(\Ad_g\) of \(\stlin(R)\) such that \(\stmap\) is \(g\)-equivariant, because the extension \(\stlin(R) \rar \elin(R)\) is central and perfect. If we show its existence, then \(\Ad_{gh} = \Ad_g \circ \Ad_h\) by uniqueness. Let \(Y_{ij}(a) = \stmap^{-1}(\up g{t_{ij}(a)})\), these objects are certain cosets of \(\stlin(R)\) by \(\klin_2(R)\). They satisfy the Steinberg relations as follows:
\begin{enumerate}[label = (St\arabic*)]
\item \(Y_{ij}(a)\, Y_{ij}(b) = Y_{ij}(a + b)\);
\item \([Y_{ij}(a), Y_{kl}(b)] \subseteq \klin_2(R)\) for \(j \neq k\) and \(i \neq l\);
\item \([Y_{ij}(a), Y_{jk}(b)] \subseteq Y_{ik}(ab)\) for \(i \neq k\).
\end{enumerate}
Since \(\stlin(R) \rar \elin(R)\) is a central extension, all commutators \([Y_{ij}(a), Y_{kl}(b)]\) are one-element sets. We identify them with their elements. Let us show that actually \([Y_{ij}(a), Y_{kl}(b)] = 1\) for \(j \neq k\) and \(i \neq l\). Fix the indices \(i, j, k, l\) and an element \(a \in R_{ij}\). Consider the ideal
\[\mathfrak c = \{k \in K \mid [Y_{ij}(a), Y_{kl}(kR_{kl})] = 1\} \leqt K.\]
We show that \(\mathfrak c\) is not contained in any maximal ideal \(\mathfrak m \leqt K\). By proposition \ref{LinearLocalAction}, the element \(\Psi(\up g{t_{ij}(a)}\, g) = \Psi(g)\, \Psi(t_{ij}(a))\) acts on \(\stlin(R)^{(\infty)}\) and \(\Psi(t_{ij}(a))\) acts as in lemma \ref{SteinbergLocalAction}. It follows that \([Y_{ij}(a), Y_{kl}(kR_{kl})] = 1\) for some \(k \in K \setminus \mathfrak c\).

Now let \(y_{ijk}(a \otimes b) = [Y_{ij}(a), Y_{jk}(b)] \in Y_{ik}(ab)\) for \(a \in R_{ij}\) and \(b \in R_{jk}\) if \(i, j, k\) are distinct. The elements \(y_{ijk}(a \otimes b) \in \stlin(R)\) are biadditive:
\begin{align*}
y_{ijk}((a + a') \otimes b) &= [Y_{ij}(a + a'), Y_{jk}(b)]\\
&= [Y_{ij}(a) Y_{ij}(a'), Y_{jk}(b)]\\
&= [Y_{ij}(a), Y_{jk}(b)]\, [Y_{ij}(a'), Y_{jk}(b)]\\
&= y_{ijk}(a \otimes b)\, y_{ijk}(a' \otimes b),
\end{align*}
and similarly for the second variable.

Also for different \(i, j, k, l\) and for all \(a \in R_{ij}\), \(b \in R_{jk}\), \(c \in R_{kl}\) we have
\begin{align*}
y_{ijl}(a \otimes bc) &= [Y_{ij}(a), [Y_{jk}(b), Y_{kl}(c)]]\\
&= [[Y_{ij}(a), Y_{jk}(b)], Y_{kl}(c)\, Y_{jl}(bc)]\\
&= y_{ikl}(ab \otimes c).
\end{align*}
Since \(n \geq 4\), it follows that \(y_{ijk}(ab \otimes c) = y_{ijk}(a \otimes bc)\) for \(a \in R_{ij}\), \(b \in R_{jj}\), \(c \in R_{jk}\) by Morita equivalence of \(e_j\) and \(e_l\) for some new index \(l\) (see lemma \ref{RingGeneration} in the case \(S = \{1\}\)). For all \(i \neq k\) choose an index \(j\) different from \(i, k\) and define \(y_{ik} \colon R_{ik} \rar \stlin(R)\) by \(y_{ik}(ab) = y_{ijk}(a \otimes b)\), this is well-defined by lemma \ref{RingPresentation}. The relations (St1) and (St2) for \(y_{ij}\) are trivial, the relation (St3) follows by Morita equivalence (or lemma \ref{RingGeneration}). Let \(\Ad_g(x_{ij}(a)) = y_{ij}(a)\), this is the endomorphism of \(\stlin(R)\) making \(\stmap\) \(g\)-invariant.

Finally, note that the action \(\Ad\) of \(\glin(R)\) on \(\stlin(R)\) is consistent with the action of \(\stlin(R) \rtimes \diag(R)\) by uniqueness. Hence \(\stmap\) is a crossed module.
\end{proof}

It is known that \(\stlin(R)\) is centrally closed for any unital ring \(R\) with a complete family of Morita equivalent orthogonal idempotents \(e_1, \ldots, e_n\) for \(n \geq 5\). The proof in the matrix case is written, for example, in \cite{Milnor}, theorem 5.10 and it actually works in the general case. Hence under the assumptions of theorem \ref{QuasifiniteCrossedModule} the group \(\stlin(R)\) is the universal central extension of \(\elin(R)\) for \(n \geq 5\).

\bibliographystyle{plain}  
\bibliography{references}

\end{document}